\documentclass[11pt,a4paper]{article}
\usepackage[utf8]{inputenc}
\usepackage[T1]{fontenc}
\usepackage[english]{babel}
\usepackage{csquotes}
\usepackage{amsmath, amssymb, amsthm}
\usepackage{mathrsfs}
\usepackage{geometry}
\usepackage[pdftex]{hyperref}
\usepackage{float}
\usepackage{enumitem}
\usepackage{xcolor}
\usepackage{graphicx}
\usepackage{tikz}
\usepackage{stmaryrd}
\usepackage[most]{tcolorbox}
\usepackage{thmtools}
\usepackage{lmodern}
\usepackage[
  backend=biber,
  style=numeric,
  sorting=none   
]{biblatex}
\usepackage{thm-restate}
\theoremstyle{plain} 
\newtheorem{theorem}{Theorem}[section]
\newtheorem{corollary}[theorem]{Corollary}
\newtheorem{proposition}[theorem]{Proposition}
\newtheorem{lemma}[theorem]{Lemma}
\newtheorem{definition}[theorem]{Definition}
\newtheorem{example}[theorem]{Example}
\newtheorem{remark}[theorem]{Remark}

\newcommand{\E}{\mathbb{E}} 
\newcommand{\Prob}{\mathbb{P}} 
\newcommand{\N}{\mathbb{N}} 
\newcommand{\Z}{\mathbb{Z}} 
\newcommand{\R}{\mathbb{R}} 
\newcommand{\eps}{\varepsilon} 

\title{Criticality of the abelian sandpile in dimension 1}
\author{Adrien REZZOUK\thanks{École normale supérieure, PSL University, 75005 Paris, France. 
Email: adrien.rezzouk@ens.psl.eu}}
\date{\today}

\usepackage{filecontents}

\begin{document}

\maketitle

\begin{abstract}
We study the critical behavior of the dissipative abelian sandpile model on  $\Z$ with dissipative sites at arbitrary positions $(x_k)$. This is equivalent to studying whether the expected stopping time of a trapped random walk on $\Z$ is finite. 
Our main contribution is a precise description of this phase transition via three distinct regimes.
Our analysis captures criticality through the asymptotic behavior of an explicit recursive sequence, revealing counterintuitive phenomena in which traps may be farther apart yet the stopping time becomes shorter.
\end{abstract}

\section{Introduction}

A walker wanders on $\Z$ but deadly traps are dispersed at fixed sites $x_k \in \Z$. When he encounters a trap, he may escape with probability $2/3$; otherwise his journey ends. How long can he survive? This question is equivalent to determining whether the Abelian sandpile model~—~a toy model introduced by physicists Per Bak, Chao Tang, and Kurt Wiesenfeld in \cite{BTW} and later generalized by Deepak Dhar in \cite{Dhar_1999}~—~exhibits critical behavior.

More specifically,  we consider the Abelian sandpile model on the graph \(\Z\), where certain sites, called dissipative or trap vertices are located at positions \(x_k \in \Z\). Each such site is connected to an additional vertex — the sink — through a third edge, allowing grains of sand to be lost. These trap vertices thus have degree 3. Whenever the height at such a vertex exceeds 3, it topples by sending one grain to the left, one to the right, and one into the sink. The dynamics continues until all vertices are stable, which occurs when the system is confined to a finite volume, e.g., \(\llbracket -n, n \rrbracket\), and the boundary sites are connected to the sink.

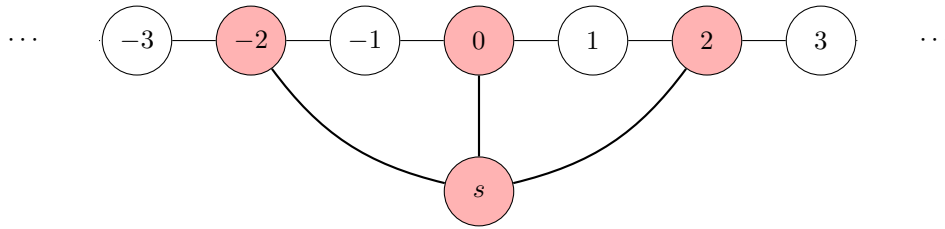
\begin{figure}[ht]
        \centering
        \begin{tikzpicture}[scale=1,
            every node/.style={circle, draw, minimum size=26pt, font=\small},
            dissip/.style={fill=red!30},
            normal/.style={fill=white}
            ]

            \foreach \i/\lab in {-3/, -2/, -1/, 0/, 1/, 2/, 3/} {
                \ifnum\i=-2 \def\sty{dissip}\else
                \ifnum\i=0  \def\sty{dissip}\else
                \ifnum\i=2  \def\sty{dissip}\else
                \def\sty{normal}
                \fi
                \fi
                \fi
                \node[\sty] (z\i) at (\i*1.5,0) {\lab  $\i$};
            }

            \node[dissip] (p) at (0, -2) {$s$};

            \foreach \i in {-3,-2,-1,0,1,2} {
                \pgfmathtruncatemacro{\j}{\i + 1}
                \draw[-] (z\i) -- (z\j);
            }


            \draw[thick] (z-2) to[bend right=20] (p);
            \draw[thick] (z0) -- (p);
            \draw[thick] (z2) to[bend left=20] (p);

            \node[draw=none] at (-6,0) {\(\cdots\)};
            \node[draw=none] at (6,0)  {\(\cdots\)};
            \draw[dotted] (-5,0) -- (z-3);
            \draw[dotted] (z3) -- (5,0);


        \end{tikzpicture}

        \caption{Example with traps in \(x_{-1} = -2\), \(x_0 = 0\) et \(x_1 = 2\)}

    \end{figure}

By adding grains of sand at uniformly chosen sites, one defines a Markov chain on the set of height configurations \(\eta\) over the graph. In finite volume, this chain admits a unique invariant measure, denoted \(\mu_n\) for the volume $\llbracket -n,n\rrbracket$, which is the uniform distribution over the set of recurrent configurations. For a concise yet precise introduction to the Abelian sandpile model, see, for instance~\cite{mathematicalintroduction}; for a more comprehensive treatment covering additional aspects, see~\cite{jaraisandpilemodel}.

The dissipative version of the Abelian sandpile model is analyzed in \cite{Redig2018-gl}, where the authors introduce a notion of criticality on \(\Z^d\). For a given configuration \(\eta\) and site \(x\), let \(Av(x, \eta)\) denote the number of sites involved in the avalanche triggered by adding a grain at \(x\). The model is said to be non-critical on $\Z$ if
\[
\forall x \in \Z, \quad \limsup_{n \to \infty} \E_{\mu_n}[Av(x, \eta)] < \infty.
\]

The transition between critical and non-critical behavior in one dimension is further explored in \cite{redig2025randomwalksfieldsoft}. There, the authors use the characterization of the criticality of the model in terms of the divergence of a stopping time associated with a trapped random walk on~\(\Z\). Their analysis focuses on the specific case where the traps are symetrically distributed and the distances between traps, denoted \(\lvert I_k \rvert = x_k - x_{k-1}\), satisfy the recursive relation \(\lvert I_{k+1} \rvert = c \lvert I_k\rvert^2 \) for $k\geq0$. They show that the model exhibits critical behavior if and only if \(c \geq 1\) by deriving bounds the expected stopping time.

In this work, we consider arbitrary sequences \((x_k)\) of trap positions and compute the expectation of the above stopping time of the random walk. We first reduce the problem of criticality to the analysis of the trapped random walk on the half-line $ \mathbb{Z}_{>0}$, where the traps are located at positions $x_k > 0$.
Our analysis reveals several phenomena that challenge intuition. First, while one might expect that spacing traps further apart should always increase escape times, we discover situations where \emph{doubling all trap distances can turn an infinite expected stopping time into a finite one}. 

More fundamentally, we identify three distinct regimes separated by two natural thresholds. The first threshold appears at the geometric sequence \(\left(\frac{3+\sqrt{5}}{2}\right)^k\) — the square of the golden ratio — below which all trap configurations yield non-critical behavior. Above this lies an intermediate regime where critical and non-critical systems coexist, with a decisive criterion being whether the ratio \(x_{k+1}/x_k^2\) consistently stays above or below 1 (which also establishes that recursive constructions such as \(\lvert I_{k+1} \rvert = c \lvert I_k\rvert^2\), studied in \cite{redig2025randomwalksfieldsoft}, capture this transition).
Finally, any sequence that eventually dominates all double-exponentials \(a^{2^k}\) with \(a > 1\) necessarily produces critical behavior. To do so, we restrict the walk to a finite interval \([-x_n, x_n]\) and compute the expected stopping time of this truncated walk. We then analyze its behavior as \(n \to \infty\) in order to determine whether the expected stopping time diverges, which corresponds to the critical regime.

The paper is organized as follows. In Section~\ref{sec:mainresults}, we state our main results, and give various examples to illustrate the possible cases. Then, we begin by formulating the problem in terms of the Green function of the graph in Section~\ref{sec:walk}. We then prove that the criticality of the Abelian sandpile model is equivalent to the divergence of the expected value \(\E_x(T)\), where \(T\) denotes the stopping time of a random walk on \(\Z\) in which the walker, when located at a trap site $x_k$, jumps with equal probability to the left, to the right, or into the sink.

In Section~\ref{sec:framework}, we derive the main equation on $\E_x(T)$ that will be solved in Section~\ref{sec:q_n/p_n}. We also prove that the criticality of the model is preserved under modifications of finitely many traps, and we show how the study of the walk on $\mathbb{Z}$ reduces to the study of the walks on the half-lines $\mathbb{Z}_{<0}$ and $\mathbb{Z}_{>0}$.

Finally, in Section~\ref{sec:q_n/p_n}, we express the criticality in terms of the limit of an explicit sequence, establishing this three-regime structure and providing practical criteria to determine the critical nature of many trap configurations.
\section{Main results}
\label{sec:mainresults}
All the results in this section are proved in sections~\ref{sec:walk},~\ref{sec:framework} and~\ref{sec:q_n/p_n}.
As mentioned earlier, the problem can be reformulated as a random walk on $\mathbb{Z}$.  
Let $D = \{x_k\}_{k \in \mathbb{Z}}$ denote the set of \emph{traps}, and introduce a special absorbing state, the \emph{sink} $s$. Once the walker reaches $s$, it cannot leave; equivalently, the walk terminates.  

The transition probabilities are given by, for $x \in \Z$, $y \in \Z\cup\{s\}$  
\begin{equation}
\label{eq:P(x,y)}
P(x, y) =
\begin{cases}
0 & \text{if } x = y, \\
\frac{1}{2} & \text{if } x \sim y \ \text{and} \ x \notin D, \\
\frac{1}{3} & \text{if } x \sim y \ \text{and} \ x \in D,
\end{cases}
\end{equation}
where $x \sim y$ means that $y$ is a nearest neighbor of $x$, i.e., 
\[
y = x+1, \quad y = x-1, \quad \text{or} \quad y = s \ \text{if} \ x \in D.
\]

Thus, if $x \notin D$, the walker moves to the left or to the right with probability $1/2$ each;  
if $x \in D$, it moves left, right, or into the sink $s$ with probability $1/3$ each.

Let $T$ be the stopping time of the walk, i.e., the smallest time at which the walker reaches the sink $s$. Denote by $\mathbb{E}_x(T)$ the expected stopping time when starting from site $x$.  
We first establish that the model is critical in the sense of \cite{Redig2018-gl} precisely when the expected stopping time of the associated random walk is infinite:
\begin{theorem}
\label{thm:equivalencewalk}
The model is critical if and only if, for every $x \in \mathbb{Z}$, $\mathbb{E}_x(T) = \infty$,  
which is equivalent to the existence of some $x \in \mathbb{Z}$ such that $\mathbb{E}_x(T) = \infty$.
\end{theorem}

To compute the expected stopping time, we first study the trapped random walk on the bounded domain $V_n = \llbracket -n, n \rrbracket$.
\begin{definition}
\label{def:T_n}
We define $T_n$ to be the first time the walker reaches $s$, $-(n+1)$, or $(n+1)$.
\end{definition}

\begin{remark}
This is equivalent to considering the random walk on $V_n \cup \{s\}$ with transition probabilities given by $P(x,y)$ in \eqref{eq:P(x,y)} for $x,y \in V_n$, where the probabilities for $y=s$ are then automatically determined. The sink is an absorbing state ($P(s,s) = 1)$.
This point of view will be adopted in section~\ref{sec:walk}.
\end{remark}
Then $\mathbb{E}_x(T_n)$ is always finite, and 
\[
\mathbb{E}_x(T) = \lim_{n \to \infty} \mathbb{E}_x(T_n).
\]
Note that in Section~\ref{sec:q_n/p_n}, we define $T_N$ to be the stopping time on the domain $V_{x_N} = \llbracket -x_N, x_N \rrbracket$.  
In this case as well, we have
\[
\mathbb{E}_x(T) = \lim_{N \to \infty} \mathbb{E}_x(T_N).
\]
Finally, do not confuse $T_n$ and $T_N$ with $T_x$ the stopping time of the walk on $\Z$ associated with the trap sequence $\{x_k\}$ defined in subsection~\ref{subsec: monotonicity}.

\subsection{Overview of the main results}
We consider the trapped random walk with an arbitrary set of traps $D$.

The criticality of the full system is completely determined by the criticality of the corresponding half-line processes. By this, we mean the trapped random walk on $\Z_{>0} =\{1,2,3,\dots\}$ (resp.~$\Z_{<0}=\{\dots,-3,-2,-1\}$) with the same transition probabilities as the original walk, which is stopped upon reaching either the absorbing state $s$ or the boundary point $0$.

More precisely, the full system is critical if and only if at least one of these two half-line processes is critical (see Theorem~\ref{thm:excursion_equivalence} for the mathematical formulation).
\begin{remark}
The choice of $0$ as the reference point is not essential. By Theorem~\ref{thm:equivalencewalk}, the criticality of the system does not depend on the starting point $x \in \Z$. Therefore, the integer line $\Z$ can be split into two halves at any point. 
\end{remark}

Consequently, we can focus without loss of generality on the half-line $\Z_{>0}$. We denote
\[
D_{+} = D \cap \Z_{>0} = \{x_k\}_{k \geq 1}, \qquad 0 < x_1 < x_2 < \dots,
\]
and study the random walk with traps at positions $(x_k)_{k \geq 1}$. We use the convention $\mathbb{N} = \{0,1,2,\dots\}$. In the following, when we refer to the model or system as critical or non-critical, we mean the walk on the half-line. For the full walk on $\Z$ (or the original abelian sandpile), criticality occurs as soon as at least one of the half-line walks is critical.

\begin{remark}
In the special case of a symmetric distribution of traps (i.e., $x \in D$ if and only if $-x \in D$), the criticality of the walk on $\Z$ is equivalent to the criticality of the walk on $\Z_{>0}$.
\end{remark}

\begin{restatable}{theorem}{Sequence}
    \label{thm:q_np_n}
Let $D_+ = \{x_k\}_{k \geq 1}$ be the ordered sequence of trap positions on $\Z_{>0}$. 

Define the sequences $(p_n)_{n \ge1}$ and $(q_n)_{n \ge 1}$ by the recursions
\begin{equation}
\label{eq:defp_nq_n}
\left\{
\begin{aligned}
p_1 &= x_1, \\
p_n &= x_n + \sum_{i=1}^{n-1} (x_n - x_i) p_i, & \text{for } n \geq 2,
\end{aligned}
\right.
\quad
\left\{
\begin{aligned}
q_1 &= x_1^2, \\
q_n &= x_n^2 + \sum_{i=1}^{n-1} (x_n - x_i) (q_i+1), & \text{for } n \geq 2.
\end{aligned}
\right.
\end{equation}

Then, $\dfrac{q_n}{p_n}$ is a positive increasing sequence, and the model is critical if and only if
\[
\lim_{n \to \infty} \frac{q_n}{p_n} = \infty.
\]
\end{restatable}

By analyzing this sequence, we find the following cases where we know whether the model is critical or non-critical. 
\begin{restatable}{theorem}{Weaknonc}
\label{Weaknonc}
    If \(x_k = \mathrm{O}(\rho^k)\) with \(\rho < \dfrac{3+\sqrt{5}}{2}\), then the system is non-critical.
\end{restatable}

The geometric sequence with ratio \( \frac{3 + \sqrt{5}}{2} \) acts as a threshold: sequences growing more slowly always yield critical systems, whereas slightly faster ones may be non-critical.

\begin{restatable}{theorem}{Weaknonctwo}
\label{Weaknonctwo}
For any increasing sequence of integers \( (y_k) \) such that
\[
k \left( \frac{3 + \sqrt{5}}{2} \right)^k = o(y_k) \quad \text{as } k \to \infty,
\]
there exists a trap sequence \( (x_k) \) such that the model is critical and \( x_k \leq y_k \) for all \( k \geq 1 \).

\end{restatable}

\begin{remark}
    \label{remark:counterexample}
    For all $\rho >\frac{3 + \sqrt{5}}{2}  $, define $y_k = \lfloor \rho ^k \rfloor$. Then we can apply Theorem~\ref{Weaknonctwo} to find a trap sequence $(x_k) $ such that the model is critical and \( x_k \leq \rho^k \) for all \( k \geq 0 \). This shows that Theorem~\ref{Weaknonc} does not extend beyond the threshold 
    \( \frac{3 + \sqrt{5}}{2} \), since for larger $\rho$ one can construct counterexamples.
\end{remark}

Above \( \left( \dfrac{3 + \sqrt{5}}{2} \right)^k \), critical and non-critical coexist. When the sequence is regular enough, in the sense that \( \frac{x_{n+1}}{x_n^2} \) stays consistently on one side of 1 (either always below or always above), then we can determine whether the system is critical or not:
\begin{restatable}{theorem}{Strongnonc}
\label{thm:Notc}
    If \(\limsup\limits_{n\to \infty} \dfrac{x_{n+1}}{x_n^2}<1\), then the model is non-critical. 
\end{restatable}

On the contrary:
\begin{restatable}{theorem}{Strongc}
\label{thm:Strongc}
   If $\displaystyle \sum_{n=1}^{\infty} \left(x_n\prod_{r=1}^{n-1} \frac{1}{1 + x_r} \right) = \infty$, then the model is critical. 
\end{restatable}
\begin{remark}
    This condition is not necessary: Theorem~\ref{Weaknonctwo} provides counterexamples. 
    
    If $x_k \leq \rho^k$, then the series $\displaystyle \sum_{n=1}^{\infty} \left(x_n\prod_{r=1}^{n-1} \frac{1}{1 + x_r} \right)$ converges, because $\displaystyle \prod_{r=1}^{n-1} \frac{1}{1 + x_r} \leq \frac{1}{(n-1)!}$.
\end{remark}
\begin{restatable}{corollary}{Weakc}
\label{cor:weakc}
    If $\liminf\limits_{n\to \infty} \dfrac{x_{n+1}}{x_n^2}>1$, then the model is critical.
\end{restatable}
\begin{restatable}{corollary}{Weakcthreshold}
\label{Weakcthreshold}
If, for every \(a > 1\), there exists \(k \ge 1\) such that \(x_k > a^{2^k}\), then the system is critical.
\end{restatable}
\begin{remark}
    The sequences $a^{2^k}$ thus form the boundary between sequences that may yield a non-critical system and those that necessarily yield a critical system.\\
    Summary: if \(x_k \ll \left( \frac{3+\sqrt{5}}{2} \right)^k\), the system is non-critical.  
If \(x_k \gg a^{\,2^k}\), the system is critical.  
In all other cases, both regimes coexist; a practical test consists in examining the ratio \(\dfrac{x_{k+1}}{x_k^2}\).

\end{remark}

\begin{example}
\label{ex:noncritical}
Each of the following sequences of traps is associated with a non-critical model:
\begin{itemize}
  \item \( x_k \sim a^{b^k} \), with \( a > 1 \), \( 1 < b < 2 \).

  \item \( x_k \sim \lambda a^{2^k} \), with \( a > 1 \), \( \lambda > 1 \).
\end{itemize}
\end{example}

\begin{example}
\label{ex:critical}
Each of the following sequences of traps is associated with a critical model:
\begin{itemize}
  \item \( x_k \sim a^{b^k} \), with \( a > 1 \) and \( b > 2 \).

  \item \( x_k \sim \lambda a^{2^k} \), with \( a > 1 \) and \( \lambda < 1 \).

  \item \( x_k = a^{2^k}(1 + \varepsilon_k) \) with \(a>1\) and \( \sum_{k=1}^\infty |\varepsilon_k| < \infty \).
\end{itemize}
\end{example}
\begin{remark}
    These examples make it possible to recover results from~\cite{redig2025randomwalksfieldsoft}: the recurrence relation on \(\lvert I_k \rvert \)
    \[
        |I_{k+1}| = c\,|I_k|^2,
    \]
    yields the explicit form
    \[
       \lvert I_k \rvert= \frac{1}{c}\,(c\lvert I_0\rvert)^{2^k}.
    \]
    Since \(\lvert I_k \rvert= x_k - x_{k-1}\), we obtain
    \[
        x_k = \sum_{j=1}^k \lvert I_j \rvert \ \sim\ \lvert I_k \rvert\quad \text{as } k \to \infty.
    \]
    Setting \(\lambda = \frac{1}{c}\), these examples show that the system is \emph{critical} for \(c \geq 1\) and \emph{non-critical} for \(c < 1\).
\end{remark}

\begin{remark}
    The method exposed here can be used to solve a more general case where the probability to fall in a trap is \(p \in (0,1)\).
    Let \(\alpha = \frac{2p}{1-p}\), if \(\limsup \frac{x_{n+1}}{x_n^2}<\alpha\), the model is non-critical and if \(\liminf \frac{x_{n+1}}{x_n^2}>\alpha,\) the model is critical.
\end{remark}

\subsection{Remarks on monotonicity}
\label{subsec: monotonicity}
We consider two sequences of traps \((x_k)\) and \((y_k)\), and denote by \(T_x\) (resp. \(T_y\)) the associated stopping times. We would like to compare \(\E(T_x)\) and \(\E(T_y)\) if we have some comparison between \((x_k)\) and \((y_k)\).

\begin{proposition}
    If \(\{x_k\} \subset \{y_k\}\), then \(T_x \geq T_y\).
\end{proposition}

\begin{proof}
    This is clear from Proposition~\ref{prop:Decorrelation walk}: the random walk with traps placed at \(\{y_k\}\) encounters more traps than the one with traps at \(\{x_k\}\).
\end{proof}

We would have liked a simple characterization of criticality in terms of the asymptotic analysis of the sequence \(x_k\), that is, ideally a property like \textit{``if \(x_k = o(y_k)\) and \(\E(T_y) < \infty\), then \(\E(T_x) < \infty\)''}.
This would follow from a monotonicity of the expectation of the form \textit{``if \(x_k \leq y_k\), then \(\E(T_x) \leq \E(T_y)\)''}.
However, this hypothesis, although intuitive, turns out to be false due to the previously exposed results:

\textbf{Counterexample:} we have seen that if \(x_k \sim \lambda a^{2^k}\), then if \(\lambda > 1\), \(\E(T_x) < \infty\), whereas if \(\lambda \leq 1\), \(\E(T_x) = \infty\).
Thus, by taking for example \(\lambda_1 = 1\) and \(\lambda_2 = 2\), \textbf{one passes from an infinite to a finite expectation even though the traps have been more widely spaced (by a homothety)}.\\

Our intuition is challenged because one might expect that spacing out the traps would increase the time to get from one trap to the next, thus increasing the stopping time.
However, adopting the point of vue of Proposition~\ref{prop:Decorrelation walk}, we may understand that sometimes spacing out the traps places them in strategic positions that can stop trajectories that would otherwise have been very long.\\

What is even more remarkable is that for all $\rho > \frac{3+\sqrt{5}}{2}$ there exists a sequence of traps \( (x_k)_{k \ge 1} \) such that \( x_k \leq \rho^{k } \), for which the model remains \textbf{critical} (see remark~\ref{remark:counterexample}). Observe that \( x_k \ll a^{b^k} \) for all \( a > 1 \), \( b > 1 \); that is, this sequence grows much more slowly than any double-exponential sequence, and is thus negligible compared to the non-critical cases discussed above.

\section{Definition of criticality and connection with the trapped random walk}
\label{sec:walk}

Define $D = \{ x_k \}_{k \in \Z}$ the set of traps.

We can then define the graph's Laplacian matrix and Green's function:
\begin{definition}[Laplacian of the graph]

The Laplacian matrix \( \Delta \) is defined for all \( x, y \in \Z \) by :
\begin{equation}
\Delta_{x,y} = 
\begin{cases}
\deg(x) & \text{si } x = y, \\
-1 & \text{si } x \sim y, \\
0 & \text{sinon},
\end{cases}
\end{equation}
where the degree is given by:
\[
\deg(x) = 2 + \mathbf{1}_{x \in D}.
\]

\end{definition}
\begin{definition}[Green's function]
Let \(V_n = \llbracket -n, n \rrbracket \) denote the restricted graph. The Green's function on this graph is defined as:
\begin{equation}
G_n(x, y) = \Delta_n^{-1}(x, y),
\end{equation}
where \( \Delta_n \) is the discrete Laplacian matrix restricted to the vertices in \( \llbracket -n, n \rrbracket \).
\end{definition}

\begin{remark}
We have \( G_n(x,y) = \E(n(x,y,\cdot)) \) (see e.g. \cite{jaraisandpilemodel}), where expectation is taken according to the invariant measure of the Markov chain, and where \( n(x,y,\eta) \) denotes the number of grains that the vertex \( y \) emits before stabilization when a grain of sand is added to \( x \) from the configuration \( \eta \).
\end{remark}

\begin{definition}[Criticality]

    We define $\displaystyle G(x,y) = \lim_{n \to \infty} \uparrow G_n(x,y) = \sup_{n \in \N} G_n(x,y) \in [0,\infty]$ whose existence will be justified later.
    Then, as defined in \cite{Redig2018-gl}, the model is said to be \textbf{non-critical} if for all $x \in \Z$: \[ \sum_{y\in \Z} G(x,y) = \lim_{n \to \infty}  \sum_{y\in \Z} G_n(x,y) < \infty.\]

    Otherwise, the model is said to be \textbf{critical}.
\end{definition}

\begin{remark}
The model is \emph{non-critical} if, and only if, for any $x \in \Z$, the average size of an avalanche triggered in $x$ is finite, as specified in \cite{Redig2018-gl}.
\end{remark}

\textbf{Connection with the trapped random walk defined in Section~\ref{sec:mainresults}.}  

Recall the random walk $(X_n)_{n \in \N}$ on \(V_n \cup \{s\}\) introduced in Section~\ref{sec:mainresults}, with stopping time \(T_n\) upon reaching the sink \(s\). Its transition matrix \(P\) from \eqref{eq:P(x,y)} can be written as
\[
P(x, y) = \delta_{x, y} - \frac{\Delta_n(x, y)}{\deg(x)},
\]
where \(\Delta_n\) is the graph Laplacian on \(V_n\) and \(\deg(x)\) is the degree of \(x\in V_n\).

Let $\E_x( \cdot)= \E(\cdot \mid X_0=x).$ We can write 
 \(\Delta_n = \Lambda (I_n-P)\) where $\Lambda$ is a diagonal matrix with $\Lambda(x,x) = \deg(x)$. Hence, if the series of the right-hand term converges:
\begin{equation}G_n = (I_n -P)^{-1} \Lambda^{-1} = (\sum_{k =0}^{\infty}P^k)\Lambda
^{-1}.\end{equation}

We can define $Q_n(x,y)= \displaystyle \sum_{k =0}^{\infty}(P^k)_{x,y} \in [0,\infty] $.

Thus $Q_n(x,y) = \E_x(\text{number of times $y$ is reached in the random walk})$.
Then,
\begin{equation}
\sum_{y \in V_n} Q_n(x,y) = \E_x(T_n).
\end{equation}

\noindent
It is then clear from Markov's property that $\E_x(T_n)$ is finite, so $\sup_{x \in V_n} \sum_{y \in V_n}  Q_n(x,y)< \infty $ which guarantees that the $P^k$ series converges.

Finally:
\begin{align*}
 G_n(x,y)
&= \frac{1}{\deg(y)}\sum_{k =0}^{\infty}(P^k)_{x,y} \\\
&= \frac{Q_n(x,y)}{\deg(y)}\\
&= \frac{1}{\deg(y)} \, \E_x\big(\text{number of reaches of } y \text{ during random walk on } V_n\big).
\end{align*}

$G_n(x,y)$ is therefore increasing with $n$: we can couple the random walks on the different $V_n$ by considering the global random walk on $\Z\cup\{s\}$ so that the walk on $V_n$ stops when the global random walk falls in $s$ or exits $\llbracket -n,n \rrbracket$.\\

This justifies the existence of $G(x,y)$. We can now express the criticality of the model from the random walk on $\Z$:

\[\frac{1}{3} \E_x(T_n) \leq \sum_{y\in V_n} G_n(x,y) = \sum_{y\in V_n} \frac{Q_n(x,y)}{\deg(y)} \leq \frac{1}{2} \E_x(T_n).\]

Let $T= \displaystyle \lim_{n \to \infty} \uparrow T_n$ the time to reach $s$ for the global random walk on $\Z \cup \{s\}$. We thus obtain the following characterization:

\begin{proposition}

The model is non-critical if and only if for every $x \in \Z$,
$
\ \E_x(T) < \infty.
$

\begin{remark}
We will see in the next section that this is equivalent to the existence of some $x \in \Z$ such that $\E_x(T) < \infty$, which will complete the proof of Theorem~\ref{thm:equivalencewalk}.
\end{remark}
\end{proposition}

In practice, we will look at $\E(T_n)$ and then make $n$ tend towards $+\infty.$

\section{Study of the trapped random walk on \(\Z\): framework}
\label{sec:framework}
We consider a random walk on \( \Z \cup \{s\} \) where the set of dissipative sites (traps) is denoted by \( D \). The transition probabilities are given by~\eqref{eq:P(x,y)}.
Let $X_n$ be the position at time $n$ of the walk. The stopping time of the walks is defined by 
\begin{equation}
T = \inf\{n\geq1 : X_n=s\}.
\end{equation}
 It is interesting\footnote{This property is used in subsection~\ref{subsec: monotonicity}} to decorrelate the number of traps visited and the simple walk on $\Z$ via the following proposition:

\begin{proposition}
\label{prop:Decorrelation walk}
Let $\mathcal{N}$ be a geometric random variable on $\Z_{>0}$ with parameter $1/3$, and $(U_n)_{n \in \N}$ be a random walk on $\Z$ independent of $\mathcal{N}$, such that :
\[
\Prob(U_{n+1} = U_n \pm 1)=\frac{1}{2}
\]
We define the process $(V_n)_{n \in \N}$ by :
\[
V_n = 
\begin{cases}
U_n & \text{if the number of traps visited by } (U_k)_{k < n} \text{ is strictly less than } \mathcal{N}, \\
s & \text{otherwise}.
\end{cases}
\]
Then, the process \( (V_n)_{n \in \N} \) has the same distribution as \( (X_n)_{n \in \N} \).
\end{proposition}

The proof of the proposition consists simply in observing that the laws of the paths $(X_1,...,X_n)$ and $(V_1,...,V_n)$ are indeed the same.

\begin{remark}
    This property states that the random walk \((X_n)\) can be viewed either as a walk on \(\Z\) with random absorption at certain vertices called ``traps'', or as a simple walk on \(\Z\) stopped after visiting a random number of traps, independently of the path. \\
\end{remark}

The rest of the section is organized as follows:
we first complete the proof of Theorem~\ref{thm:equivalencewalk} and show that the model is critical when $D$ is finite in subsection~\ref{subsec:finiteD}. Then, in subsection~\ref{subsec:infiniteD}, we show how to reduce the question of the criticality of the walk on $\mathbb{Z}$ 
to the criticality of two walks on $\mathbb{N}$.  This reduction is crucial to obtain a recurrence relation for $\mathbb{E}_x(T)$ starting from the origin.

\subsection{Finite trap set case}
\label{subsec:finiteD}
We will show that the expectation of $T$ is infinite by introducing the method used in the general case.
Let $E(x)= \E_x(T) = \E(T \mid X_0=x) \in [ 0,+\infty ]$. By the total expectation formula and the Markov property, $E$ verifies:
\begin{equation*}
\begin{aligned}
E(x) &= \E(T \mid X_0 = x) \\
     &= \E(T \mid X_0 = x,\, X_1 = x+1)\, P(x, x+1) 
     + \E(T \mid X_0 = x,\, X_1 = x-1)\, P(x, x-1) \\
     &\quad + \E(T \mid X_0 = x,\, X_1 = s)\, P(x, s) \\
     &= \big( \E(T \mid X_0 = x+1) + 1 \big)\, P(x, x+1) 
     + \big( \E(T \mid X_0 = x-1) + 1 \big)\, P(x, x-1)
     + 1 \cdot P(x, s).
\end{aligned}
\end{equation*}

Hence $E$ verifies the discrete differential equation:
\begin{equation}
\label{eq:systemE(x)}
\begin{cases}
  E(x) = \frac{1}{2} E(x-1) + \frac{1}{2} E(x+1) + 1, & \text{if } x \notin D, \\
  E(x) = \frac{1}{3} E(x-1) + \frac{1}{3} E(x+1) + 1, & \text{if } x \in D.
\end{cases}
\end{equation}

First notice that if $E(x)$ is finite for some $x \in \Z$ then $E(x)$ is finite for all $x \in \Z$. \textbf{This concludes the proof of Theorem~\ref{thm:equivalencewalk}.}

Now, if $D$ is finite, assume by contradiction that $E(x)$ is finite for all $x \in \Z$.
Then, at $x> \max D$, E verifies the equation $E(x+1)-2E(x)+E(x-1)=-2$ whose solutions are:
\[ 
E(x) = -x^2+Ax+B, ~~(A,B)\in \R^2.
\]

But then for $x$ large enough $E(x)<0$ which is absurd. We conclude:

\begin{proposition}
 If $|D| < \infty$, then $\E(T) = +\infty$, i.e., the model is critical.
\end{proposition}

\subsection{Infinite trap set case}
\label{subsec:infiniteD}

We assume $\lvert D \rvert = \infty$. We show that the criticality of the model does not depend on a finite number of traps and that the system is non-critical if and only if two trapped random walks on $\mathbb{Z}_{>0}$ and $\mathbb{Z}_{<0}$ are non-critical.

Throughout, we write $\mathbb{P}_x$ and $\mathbb{E}_x$ for probability and expectation conditioned on $X_0 = x$.

\begin{theorem}
\label{thm:excursion_equivalence}
Let $T$ denote the absorption time in the sink $s$ for the walk started at $0$. Then
\[
\mathbb{E}_0[T] < \infty 
\quad \Longleftrightarrow \quad 
\mathbb{E}_1[\tau] < \infty \;\; \text{and} \;\; \mathbb{E}_{-1}[\tau] < \infty,
\]
where $\tau = \inf\{n \geq 0 : X_n \in \{0,s\}\}$.  

Moreover, setting
\[
q := p + \frac{1-p}{2} \, p_+ + \frac{1-p}{2} \, p_-,
\]
where 
\[
p_+ = \mathbb{P}_1(X_\tau = s), 
\qquad 
p_- = \mathbb{P}_{-1}(X_\tau = s),
\]
we have
\begin{equation}
\label{eq:absorption_formula}
\mathbb{E}_0[T] 
= \frac{p \cdot 1 + \tfrac{1-p}{2} \, \mathbb{E}_1[\tau] + \tfrac{1-p}{2} \, \mathbb{E}_{-1}[\tau]}{q}.
\end{equation}
\end{theorem}

\begin{proof}
We decompose the sample path into excursions away from $0$. Define successive stopping times:
\[
\tau_0 = 0, 
\quad 
\tau_{k+1} = \inf\{n > \tau_k : X_n \in \{0, s\} \mid X_{\tau_k} = 0\}.
\]
Let 
\[
\mathcal{N} = \inf\{k \geq 1 : X_{\tau_k} = s\}
\]
be the number of excursions before absorption.  
(Caution: $\mathcal{N}$ is not the random variable defined in Proposition~\ref{prop:Decorrelation walk}.)

Each excursion has the structure: $0 \to \{\pm 1\} \rightsquigarrow \{0, s\}$, where the first step goes to $+1$ or $-1$ with probability $\frac{1-p}{2}$ each, or directly to $s$ with probability $p$. By the strong Markov property, excursions are independent. Each excursion ends in absorption with probability
\[
q = p + \tfrac{1-p}{2}\,p_+ + \tfrac{1-p}{2}\,p_-.
\]
Therefore, $\mathcal{N}$ follows a geometric distribution on $\{1,2,\dots\}$ with parameter $q$, so $\mathbb{E}[\mathcal{N}] = 1/q$.

Let $L_k = \tau_k - \tau_{k-1}$ be the length of the $k$-th excursion. The excursions $(L_k)_{k \geq 1}$ are i.i.d. with expected length
\[
\mathbb{E}[L_1] 
= p \cdot 1 \;+\; \tfrac{1-p}{2}\,\mathbb{E}_1[\tau] \;+\; \tfrac{1-p}{2}\,\mathbb{E}_{-1}[\tau],
\]
where the first term corresponds to immediate absorption from $0$, and the other two to excursions starting at $\pm 1$.  

The total absorption time is
\[
T = \sum_{k=1}^{\mathcal{N}} L_k.
\]
By Wald’s equation (valid in $[0,\infty]$),
\[
\mathbb{E}[T] = \mathbb{E}[\mathcal{N}] \, \mathbb{E}[L_1] = \frac{\mathbb{E}[L_1]}{q},
\]
which gives formula~\eqref{eq:absorption_formula}. The equivalence
\[
\mathbb{E}_0[T] < \infty 
\;\Longleftrightarrow\; 
\mathbb{E}_1[\tau] < \infty \;\text{ and }\; \mathbb{E}_{-1}[\tau] < \infty
\]
follows directly.
\end{proof}

\begin{remark}
In \cite{redig2025randomwalksfieldsoft} the traps are assumed to be symmetrically distributed, i.e., $D = \{x_k\}_{k \in \mathbb{Z}}$ with $x_{-k} = x_k$ (in particular $x_0=0$). In this case, the random walk on $\mathbb{Z}$ can be reduced to a walk on $\mathbb{N}$ by considering the process $\lvert X_n \rvert$, which behaves as a reflected random walk at $0$ on $\mathbb{N}$ with transition probabilities
\[
\mathbb{P}(0,1) = 2/3 \qquad \mathbb{P}(x,x\pm 1) = 
\begin{cases} 
1/2 & \text{if } x \notin D, \\
1/3 & \text{if } x \in D,
\end{cases}
\]
and $P(x,s) = 1/3$ for all $x \in D\cap \N$.
We have shown above that this is equivalent to considering the walk on 
\(\{1,2,\dots\}\) with traps at positions \(x_k\) for \(k \ge 1\), 
which is absorbed upon reaching either \(0\) or the sink \(s\).

\end{remark}

\begin{corollary}
\label{cor:trap_independence}
The criticality of the model is independent of whether $0$ is a trap. More generally, adding or removing finitely many traps does not affect the criticality.
\end{corollary}

\begin{proof}
Equation~\eqref{eq:absorption_formula} does not depend on the value of $p \in [0,1)$. By Theorem~\ref{thm:equivalencewalk}, criticality does not depend on the starting point, so we may disregard a finite number of traps.
\end{proof}

This result reduces the study of criticality to the two half-lines: it depends solely on the behavior of the trapped random walks on $\mathbb{Z}_{>0}$ and $\mathbb{Z}_{<0}$, considered separately. The system is non-critical precisely when both half-line processes have finite expected absorption times. In the next section, we analyze the case of $\mathbb{Z}_{>0}$, noting that the case of $\mathbb{Z}_{<0}$ follows by symmetry under $x \mapsto -x$.

\section{Expression of expected stopping time and criticality}
The previous theorem allows us to reformulate the criticality in terms of the stopping time of a random walk on $\Z_{>0}$ with probability transition given by \eqref{eq:P(x,y)}.\\

In subsection~\ref{subsec:expressionq_np_n}, we solve the system for $\mathbb{E}_x(T_N)$ with $N$ fixed, which yields Theorem~\ref{thm:q_np_n}. In subsection~\ref{subsec:studyofsequence}, to estimate the sequence $\frac{q_n}{p_n}$, we express it as a barycenter of $x_0,\dots,x_n$. This expression can then be bounded to obtain Theorems~\ref{Weaknonc} and~\ref{Weaknonctwo} in subsection~\ref{subsec:threshold}. Finally, we prove Theorems~\ref{thm:Notc} and~\ref{thm:Strongc}, which confirm the observation in~\cite{redig2025randomwalksfieldsoft} that the ratio $\frac{x_{n+1}}{x_n^2}$ characterizes the criticality of the system.

\label{sec:q_n/p_n}
\subsection{Expression in terms of \(q_n/p_n\)}
\label{subsec:expressionq_np_n}
\textbf{Framework:} We study the random walk on the half-line \(\Z_{>0}\) with traps in \(D_+=\{x_k\}_{k\ge1}\). We define $x_0=0$ which plays the same role as the sink $s$.

We consider the walk truncated to the interval \(\llbracket 1,\, x_N - 1 \rrbracket\), and let \(T_N\) denote the corresponding stopping time. There are two possible interpretations of this truncated walk:
\begin{enumerate}

    \item \textbf{First interpretation:} \(T_N\) is the hitting time of \(s\) for the random walk on \(\llbracket 1,\, x_N - 1 \rrbracket\), with transition probabilities for \(x,y \in \llbracket 1,\, x_N - 1 \rrbracket\) still given by \eqref{eq:P(x,y)}:
    \[
    P(x,y) = 
    \begin{cases}
    0 & \text{if } x = y, \\
    \frac{1}{2} & \text{if } x \sim y \text{ and } x \notin D_+, \\
    \frac{1}{3} & \text{if } x \sim y \text{ and } x \in D_+. \\
    \end{cases}
    \]
    The transition probabilities towards \(s\) are then determined accordingly.

    \item \textbf{Second interpretation:} these transition probabilities can be extended to the whole of \(\N\) (imposing that \(x_0=0\) is an absorbing state), and \(T_N\) becomes the hitting time of the set \(\{x_0,\, x_N,\, s\}\).\\
\end{enumerate}

\medskip

With the second interpretation, the stopping times \(T_N\) are defined on the same space, and \(T = \lim \uparrow T_N\) gives
\[
\E_x(T)
\;=\;
\lim_{N \to \infty} \E_x(T_N).
\]

The aim of the rest of the article is to express \(\E_{x=1}(T_N)\) as a function of the trap sequence \((x_k)\).

\begin{definition}
As in Part 3.1, we define $E(x)= \E_x(T) = \E(T \mid X_0=x) \in [ 0,+\infty ]$. We also adopt the convention $E(x_0)=E(x_N)=0$.
\end{definition}

With this convention, \(E\) satisfies the system of equations~\eqref{eq:systemE(x)} for \(x \in \llbracket 1, x_N-1 \rrbracket\), which can be rewritten as:
\begin{equation}
\begin{cases}
E(x+1)-2E(x)+E(x-1)=-2 & \text{if }  x\not \in D ~(*), \\
E(x+1)-3E(x)+E(x-1)=-3 & \text{if } x \in D ~(**).
\end{cases}
\end{equation}

~

Let $E_k = E(x_k)$ where $E_0 = E(0) = 0$. The boundary conditions are $E_0 = E_N = 0$. Equation $(**)$ can be interpreted as a transmission condition on the traps.\\

We have already solved equation $(*)$:
\begin{equation} 
\label{eq:E(x)}
\forall x \in \llbracket x_k,x_{k+1} \rrbracket, ~ E(x)= -x^2+A_kx+B_k \text{ with $A_k, B_k$ real constants.}
\end{equation}

In particular, $E_0=B_0=0$.

By continuity (expression of $E(x)$ for $x=x_{k} \in \llbracket x_{k-1},x_{k} \rrbracket \cap \llbracket x_k,x_{k+1} \rrbracket$):
\[ \forall k \in \llbracket 1, n  \rrbracket, ~A_kx_k+B_k = A_{k-1}x_k + B_{k-1}.\]

Moreover, by the transmission condition \((**)\), we have
\[
E(x_{k+1}) - 3 E(x_k) + E(x_{k-1}) = -3.
\]

\noindent
We have
\[
E(x_{k}+1) - E(x_k) = -2 x_k - 1 + A_k,
\quad \text{and} \quad
E(x_{k}-1) - E(x_k) = 2 x_k - 1 - A_{k-1}.
\]
Hence,
\[
A_k - A_{k-1} = E_k - 1.
\]

Combining these two expressions, we obtain the following relations:

\[
\begin{cases}
A_k-A_{k-1}=E_k-1, \\
B_k-B_{k-1}=-x_k(E_k-1).
\end{cases}
\]

Finally: 

\[A_k =A_0 + \sum_{i=1}^k(E_i-1), ~~~B_k= -\sum_{i=1}^kx_i(E_i-1). \]

Noting $A = A_0 \in \R$ and using the relation~\eqref{eq:E(x)}, $E_k=-x_k^2+A_{k-1}x_k+B_{k-1}$ for $k\geq1$, thus, we can express:
\begin{equation}
\label{eq:E_k}
E_k=-x_k^2+Ax_k+\sum_{i=1}^{k-1}(E_i-1)(x_k-x_i).
\end{equation}

That is:
\begin{align*}
E_1 &= -x_1^2 + A x_1 = A - 1 \\
E_2 &= -x_2^2 + A x_2 + (E_1 - 1)(x_2 - x_1) \\
&\vdots \\
E_N &= -x_N^2 + A x_N + (E_1 - 1)(x_N - x_1) + \dots + (E_{N-1} - 1)(x_N - x_{N-1}).
\end{align*}

This system admits a unique solution by adding the edge condition $E_N=0$. In particular, we are interested in the variable $A=A^{(N)}=E_1+1$:\textbf{ the system is critical if and only if $\displaystyle \lim_{N \to \infty} A^{(N)} < \infty$}. \textbf{Note that $A^{(N)}$ is increasing with $N$ since the expectation of the truncated stopping time $T_N$ is}.

It naturally follows that $E_k$ can be expressed affinely in $A$: $E_k=p_kA-q_k$, so that $E_N=0$ gives $A=A^{(N)}=\frac{q_N}{p_N}$.

\textbf{Expression of $p_n$ and $q_n$:}.

\noindent
By definition: $p_1=x_1, ~q_1=x_1^2$

We define $p_k$ and $q_k$ by induction via: 
\begin{align*}
E_k&=-x_k^2+Ax_k+\sum_{i=1}^{k-1}(E_i-1)(x_k-x_i)\\
&=-x_k^2+Ax_k+\sum_{i=1}^{k-1}(p_iA-q_i-1)(x_k-x_i)\\
&= A\underbrace{\left[x_k+\sum_{i=1}^{k-1}p_i(x_k-x_i)\right]}_{p_k}-\underbrace{\left[x_k^2+\sum_{i=1}^{k-1}(q_i+1)(x_k-x_i)\right]}_{q_k}.
\end{align*}


We have proved the following:

\Sequence*

\subsection{Study of \(\displaystyle\lim_{n \to \infty} \frac{q_n}{p_n}\) }
\label{subsec:studyofsequence}
Consider the definition of $p_n$ and $q_n$ in \eqref{eq:defp_nq_n} and let $a_n=\frac{q_n}{p_n}$. 
\textbf{The model is critical if and only if $\lim_{n\to \infty}a_n = \infty$.}
Let us rewrite the recurrence relations of $(p_n)$ and $(q_n)$ as a matrix expression: we define the matrix $M \in \mathcal{M}_{n}(\R)$ by 
\begin{equation}
\label{eq:defM}
M_{i,j} =
\begin{cases}
x_i - x_j & \text{if } i > j,\\
0 & \text{otherwise,}
\end{cases}
\quad i,j \in \llbracket 1,n \rrbracket.
\end{equation}

Let
\[
X = (x_i)_{1 \leq i \leq n}, \quad X^{(2)} = (x_i^2)_{1 \leq i \leq n},  
\quad P = (p_i)_{1 \leq i \leq n}, \quad Q = (q_i)_{1 \leq i \leq n}, \quad \mathbf{1} =(1)_{1 \leq i \leq n}.
\]

Hence, the recurrence relations can be rewritten as:

\[
P=X+MP, \quad Q=X^{(2)} +M\left(Q +\mathbf{1}\right ).
\]

Since $M$ is a strict lower triangular of size $n \times n$:

\[
P=\left(\sum_{k=0}^{n-1} M^k\right) X, \quad Q=\left(\sum_{k=0}^{n-1} M^k\right)\left(X^{(2)}+\mathbf{1}\right)-\mathbf{1}.
\]

This gives us an explicit expression for $p_n$ and $q_n$, if we define $W_k^{(n)}=\sum_{m=0}^{n-1} (M^m)_{n,k}$ for $1\le k \le n$, then:

\[ p_n=\sum_{k=1}^{n} W_k^{(n)} x_k, \quad q_n= -1+\sum_{k=1}^n W_k^{(n)} (x_k^2+1).\]

Finally, denoting \(\mu_k = \frac{x_k W_k^{(n)}}{p_n}\) (which depends on \(n\)), we observe that:
\begin{equation}
\label{eq:sommemu_nvaut1}
\sum_{k=1}^n \mu_k = 1
\quad \text{and} \quad
a_n = \frac{q_n}{p_n} = \varepsilon_n + \sum_{k=1}^n \mu_k x_k,
\end{equation}
where
\[
\varepsilon_n = \frac{-1 + \sum_{k=1}^n W_k^{(n)}}{\sum_{k=1}^n W_k^{(n)} x_k}.
\]

Moreover,
\[
-1 \leq \varepsilon_n \leq 1,
\]
since \(x_k \geq 1\) and \(p_n \geq 1\).

\noindent
We set $u_n = \sum_{k=1}^n \mu_k x_k = a_n - \varepsilon_n.$ \textbf{The model is critical if and only if $\lim_{n \to \infty} u_n = \infty$}.

The goal is thus to study the barycentric coefficients \(\mu_k\) with sufficient precision.

We now give an explicit expression for \( W_k^{(n)} \). 
Since \(M\) is strictly lower triangular, we have
\begin{align*}
W_k^{(n)} &= \sum_{m=0}^{n-1} (M^m)_{n,k} \\
&= \sum_{m=0}^{n-1} \sum_{n = j_0 > j_1 > \dots > j_m = k} \prod_{r=1}^m M_{j_{r-1}, j_r} \\
&= \sum_{m=0}^{n-1} \sum_{k = i_0 < i_1 < \dots < i_m = n} \prod_{r=1}^m (x_{i_r} - x_{i_{r-1}}),
\end{align*}
where the indices are related by \(i_p = j_{m-p}\).

The term $m=0$ is equal to $0$ if $k<n$ (empty sum) and $1$ if $k=n$ (empty product).
We deduce that the \( W_k^{(n)} \)satisfy the following recurrence relation (summing according to the first intermediate index $i_1$):
\[
\forall k < n, \quad W_k^{(n)} = \sum_{j=k+1}^n W_j^{(n)} (x_j - x_k).
\]

For coefficients \( \mu_k \), this relation leads to the following identity:
\begin{equation}
\label{eq:relationmu_j}
\mu_k = x_k \sum_{j=k+1}^n \mu_j \left(1 - \frac{x_k}{x_j}\right).
\end{equation}

\textbf{Bounds:} We can now estimate the value of \(u_n=\sum_{k=1}^n\mu_k x_k\) to see whether the system is critical.
Let us define
\begin{equation}
\label{eq:defR_k}
R_k = \sum_{j=k+1}^n \mu_j,
\end{equation}
which implies \(R_{0} = 1\), $R_n=0$ and \(\mu_k = R_{k-1} - R_k.\) \textbf{Note that both the \(\mu_k\) and \(R_k\) depend implicitly on \(n\). However, we will bound them by quantities that no longer depend on \(n\).} The main challenge will then be to study the convergence of the infinite series obtained by summing these bounds from \(k = 1\) to infinity.

Observe that
\[
\sum_{k=1}^n \mu_k x_k = \sum_{k=1}^n (R_{k-1} - R_k) x_k = x_1 + \sum_{k=1}^n R_k (x_{k+1} - x_k).
\]

\begin{lemma}
\label{lemma:5.2}
    The system is critical if and only if 
    \[
    \sum_{k=1}^n R_k (x_{k+1} - x_k) \to \infty \quad \text{as } n \to \infty.
    \]
\end{lemma}

The equation~\eqref{eq:relationmu_j} will be used to estimate \( R_k (x_{k+1} - x_k) \) independently of \( n \).

\subsection{Criticality near the geometric growth threshold}
\label{subsec:threshold}

We first prove Theorem~\ref{Weaknonc} and then Theorem~\ref{Weaknonctwo}. 
For this we recall a standard description of real linear homogeneous recurrence sequences of order $k$:

Let $(\alpha_n)_{n\ge0} \in \R^\N$ satisfy
\[
a_0\alpha_n+a_1\alpha_{n+1}+\cdots+a_k\alpha_{n+k}=0\qquad(\forall n\ge0),
\]
with real coefficients $a_0,\dots,a_k$ not all zero, and let
\[
P(X)=a_0+a_1X+\cdots+a_kX^k
\]
be the characteristic polynomial. Write the distinct complex roots of $P$ as $\rho_0,\dots,\rho_p$. Then there exist (complex) polynomials $Q_j$ of degree strictly less than the multiplicity of $\rho_j$ such that
\[
\alpha_n=\sum_{j=0}^p Q_j(n)\,\rho_j^n\qquad(\forall n\ge0).
\]

Consequently, if $\rho_0$ is a simple root, $|\rho_j|<|\rho_0|$ for all $j\ne0$, and $Q_0(n)\equiv C\in\mathbb R$ with $C\neq0$, then
\[
\alpha_n\sim C\,\rho_0^n\quad\text{as }n\to\infty.
\]

To prove Theorem~\ref{Weaknonc}, we first need to prove the following technical lemma:
\begin{lemma}
\label{lemma:recurrentsequence}
Let $(\alpha_n)_{n \ge 0}$ satisfy
\[
\alpha_{n+k+1} = a_0 \alpha_n + \dots + a_k \alpha_{n+k}, \quad a_0,\dots,a_k>0,
\]
with $\alpha_0,\dots,\alpha_k>0$.  
Then the characteristic polynomial $P(X) = X^{k+1}-a_kX^k-\dots-a_0$ has a unique positive root $\rho_0$ and
\[
\alpha_n \sim C\,\rho_0^n \qquad \text{for some $C>0$.}
\]

\end{lemma}
\begin{proof}
First, let us prove that $P$ has a unique positive root $\rho_0$ and that all other roots $\rho_j$ satisfy $|\rho_j| < \rho_0$. Define $f(x) = P(x)/x^{k+1}$ for $x>0$. Then $f$ is strictly increasing, $\lim_{x\to0^+} f(x)=-\infty$, $\lim_{x\to\infty} f(x)=1$, so $f$ has a unique zero $\rho_0>0$, which is simple since $f'>0$.

If $\rho_j$ is another root, the triangle inequality gives $|\rho_j|^{\,k+1} \le a_0 + \dots + a_k |\rho_j|^k$, i.e., $f(|\rho_j|)\le0$, hence $|\rho_j|\le \rho_0$. Equality in the triangle inequality would force $\rho_j$ positive (because $a_0$ and $a_1\rho_j$ would be positively colinear), contradicting uniqueness, so $|\rho_j|<\rho_0$.

Now, write $\alpha_n = C \,\rho_0^n + \sum_j Q_j(n)\,\rho_j^n$ with $|\rho_j|<\rho_0$. We just need to show $C\neq 0$. Let $\beta_n := \alpha_n/\rho_0^n$. Then
\[
\rho_0^{k+1} = a_0 + \dots + a_k \rho_0^k, \qquad
 \beta_{n+k+1}\rho_0^{k+1} = a_0 \beta_n + \dots + a_k \beta_{n+k}\rho_0^{k}.
\]
Dividing the second equality by the first, we see that $\beta_{n+k+1}$ is a convex combination of $\beta_n,\dots,\beta_{n+k}$, hence $\beta_n \ge \min(\beta_0,\dots,\beta_k) >0$ for all $n$. Thus $C>0$, since otherwise $\beta_n$ would tend to $0$.

\end{proof}

\Weaknonc*
\begin{proof}
The goal is to find the best possible upper bound for \( R_k \) in the most general setting.  
Since \( (x_k) \) is an increasing sequence of integers, we have \( x_k \geq k \) and, for any \( j > k \),  
\(
x_j \geq x_k + (j - k).
\)
Thus, in \eqref{eq:relationmu_j}, we can bound from below the term  
\[
x_k \left(1 - \frac{x_k}{x_j}\right) \geq x_k \left(1 - \frac{x_k}{x_k + (j - k)}\right) = \frac{x_k (j - k)}{x_k + (j - k)} \geq \frac{k (j - k)}{j}.
\]

where the last inequality follows since $\frac{ab}{a+b} = \frac{b}{1 + a/b}$ decreases in $a$ for $a,b>0$.

Equation \eqref{eq:relationmu_j} then turns into (recall that $\mu_k = R_{k-1}-R_k$ and $R_n=0$):
\begin{equation}
\label{eq:inequalityR_k}
R_{k-1}-R_k \geq \sum_{j=k+1}^n \frac{k(j-k)}{j}(R_{j-1}-R_j)= \sum_{j=k}^n\frac{k^2}{j(j+1)}R_j.
\end{equation}

\noindent
The key idea is to approximate the actual relation by the idealized equality (valid for \(j \approx k \gg 1\)):
\begin{equation}
\label{eq:R_krecurrentrelation}
R_{k-1} - R_k = \sum_{j=k}^n R_j,
\end{equation}
which leads to the approximate solution
\[
R_k \approx C \, \phi^{\,n-k} R_{n-1}
\]
for some constant \(C > 0\), where \(\phi = \frac{3 + \sqrt{5}}{2}\). Indeed, defining \(Q_k = \sum_{j=n-k}^n R_j\), we observe that
\[
(Q_{k+1} - Q_k) - (Q_k - Q_{k-1}) = R_{n-k-1}-R_{n-k}=Q_k,
\]
where the second equality follows from equation~\eqref{eq:R_krecurrentrelation}. Thus,
\[
Q_{k+2} = 3 Q_{k+1} - Q_k.
\]
This second-order linear recurrence has characteristic roots \(\phi_+ = \phi\) and \(\phi_- = \frac{3 - \sqrt{5}}{2} \in (0,1)\), which implies $Q_k = \widetilde{C}\phi^k +o(1)$, $R_{k} = Q_{n-k}-Q_{n-k+1} \approx C\phi^k$ for $k\ll n$. Evaluating at \(k=-1\) gives, for \(j < n\),

\[
R_j \approx \frac{C'}{\phi^j}.,
\]
We would then conclude by lemma~\ref{lemma:5.2}: if $x_k = O(\rho^k)$, $\rho<\phi$, then $R_k x_{k+1}$ is summable.\\

To treat the real case, we introduce parameters \(0<\lambda<1\) — a uniform lower bound for \(\tfrac{k^2}{j(j+1)}\) on the indices retained in \eqref{eq:inequalityR_k} — and \(\kappa\in\mathbb{Z}_{>0}\), the number of indices kept in the sum. We will let $\lambda\to1$, and $\kappa \to \infty$ to approach the idealized case given by equation \eqref{eq:R_krecurrentrelation}.

There exists \(k_0\) such that for all \(k \geq k_0\),
\[
\frac{k^2}{(k+\kappa)(k+\kappa+1)} \geq \lambda.
\]
Hence, by inequality~\eqref{eq:inequalityR_k}, for \(k_0 \leq k \leq n\),
\[
R_{k-1} \geq R_k + \lambda \sum_{j=k}^{\min(k+\kappa,n)} R_j.
\]
Let $p \geq k_0$, $n >p$. We will bound from above $R_p$ independently of $n$. For all $k\leq p$,
\[
R_{k-1} \geq R_k + \lambda \sum_{j=k}^{\min(k+\kappa,p)} R_j.
\]
Define \((\alpha_k)\) by \(\alpha_0 = 1\) and
\[
\alpha_{k+1} = \alpha_k + \lambda \sum_{j=\max(0,k-\kappa)}^{k} \alpha_j.
\]
Then, by induction for all \(k \leq p - k_0\),
\[
R_{p-k} \geq \alpha_k R_p.
\]

The sequence \((\alpha_k)\) satisfies a homogeneous linear recurrence relation of order \(\kappa + 1\) for \(k \geq \kappa\) with characteristic polynomial
\[
P(X)= X^{\kappa + 1} - X^\kappa - \lambda (X^\kappa + \cdots + 1).
\]
We may apply Lemma~\ref{lemma:recurrentsequence}: $P$ has a unique positive root that we denote by $\phi_{\kappa,\lambda}$. Clearly, $\phi_{\kappa,\lambda}\geq 1+\lambda$ and by the lemma,
\[
\alpha_k\sim C\,\phi_{\kappa,\lambda}^k\qquad (k\to\infty)
\]
for some \(C>0\).

For all \(n \geq p\),
\begin{equation}
\label{eq:R_pinferiorto}
R_p \leq \frac{R_{k_0}}{\alpha_{p-k_0}} \leq \frac{C'}{\phi_{\kappa,\lambda}^p}
\end{equation}
for some \(C' > 0\) ($R_{k_0}\leq 1$ by \eqref{eq:sommemu_nvaut1}).

By definition, \(\phi_{\kappa,\lambda}>1\) satisfies
\[
\phi_{\kappa,\lambda}^{\kappa+1}
= \phi_{\kappa,\lambda}^\kappa
+ \lambda\frac{\phi_{\kappa,\lambda}^{\kappa+1}-1}{\phi_{\kappa,\lambda}-1},
\]
which is equivalent to
\[
\phi_{\kappa,\lambda}^2 - (2+\lambda)\,\phi_{\kappa,\lambda} + 1
= -\frac{\lambda}{\phi_{\kappa,\lambda}^\kappa}.
\]

Recall that $\phi_{\kappa,\lambda} > 1+\lambda$, therefore for \(\lambda\ge\tfrac12\),
\[
\frac{\lambda}{\phi_{\kappa,\lambda}^\kappa}
\le\frac{\lambda}{(1+\lambda)^\kappa}
\le\Bigl(\tfrac23\Bigr)^{\kappa-1}
\;\longrightarrow\;0
\quad(\kappa\to\infty).
\]
Taking \(\kappa \to \infty\) gives
\[
\phi_{\kappa,\lambda}= \frac{2 + \lambda + \sqrt{(2+\lambda)^2-4(1+\frac{\lambda}{\phi_{\kappa,\lambda}^\kappa})}}{2} \to \phi_\lambda := \frac{2 + \lambda + \sqrt{\lambda^2 + 4 \lambda}}{2}.
\]
Letting \(\lambda \to 1\), we recover \(\phi_\lambda \to \phi\).

Thus, for any \( \rho < \phi\), there exist \(\kappa, \lambda,  \rho'\) such that
\[
 \rho <  \rho' < \phi_{\kappa,\lambda}.
\]
Hence, by inequality~\eqref{eq:R_pinferiorto}, for all $p\geq k_0$,
\[
R_p \leq \frac{C'}{ \rho'^p}.
\]
Finally,
\[
\sum_{k=1}^n R_k (x_{k+1} - x_k) \leq \sum_{k=1}^n R_k x_{k+1} \leq \sum_{k=1}^{k_0-1}x_{k+1}+C' \sum_{k=k_0}^\infty x_{k+1}  \rho'^{-k}.
\]
Thus, if $x_k = O(\rho^k)$ the series converges and by lemma~\ref{lemma:5.2}, the model is non-critical.
\end{proof}

\noindent
Next, we construct counterexamples for values above the critical threshold:

\Weaknonctwo*

\begin{proof}
Let $y_k=\lambda_k \cdot k\left(\dfrac{3+\sqrt5}{2}\right)^k$ be an increasing sequence of integers with $\lambda_k \to \infty$ as $k \to \infty$.
We construct the trap sequence \( x_k \) by setting \( x_{k+1} = x_k + 1 \) for most indices, except at certain isolated steps where \( x_{k+1} \gg x_k \). At these special indices, we will have \( R_k x_{k+1} \gg 1 \). The key point is to ensure that this construction remains below \( y_k \) at every step.

We set \( x_1 =1, ~ x_2 = 2 \), and assume \( x_k \) has been defined for all \( k \leq a+1 \), with \( a \geq 1 \). Let \( b > a+1 \), which will be chosen sufficiently large later. For \( a+1<k < b \), define \( x_{k+1} = x_k + 1 \). At step \( k = b \), we will choose \( x_{b+1} \) such that \( R_b x_{b+1}  \gg 1 \).

For \(k\ge a+1\) equation \eqref{eq:relationmu_j} gives (recall that $\mu_k = R_{k-1}-R_k$)
\[
R_{k-1}-R_k=\sum_{j=k+1}^n x_k\Big(1-\frac{x_k}{x_j}\Big)(R_{j-1}-R_j).
\]
Using \(x_j=x_k+(j-k)\) for \(j\le b\) and \(\dfrac{x_k}{x_k+(j-k)}\le1\),
\[
R_{k-1}-R_k\le\sum_{j=k+1}^b (j-k)(R_{j-1}-R_j)+x_kR_b.
\]
By index manipulation,
\[
\sum_{j=k+1}^b (j-k)(R_{j-1}-R_j)=\sum_{m=k}^{b-1}R_m-(b-k)R_b\le\sum_{m=k}^{b-1}R_m,
\]
hence, since \(x_k\le x_b\),
\[
R_{k-1}-R_k  \le \sum_{m=k}^{b-1}R_m + x_bR_b.
\]

Thus, by induction for $k\leq b-a-1$ \( R_{b-k} \leq \alpha_k x_b R_b \), where \(\alpha_k\) is the sequence defined by \( \alpha_0 = 1 \) and for $k\ge0$
\[
\alpha_{k+1} = \sum_{j=0}^k \alpha_j + \alpha_k.
\]
Notice that this relation is the same that in the previous proof with $\lambda =1$ and $\kappa=\infty$

As previously noted, \( \alpha_k \sim C \phi^k \) for some constant \( C > 0 \), where \( \phi = \frac{3 + \sqrt{5}}{2} \). Indeed, the sequence \( \beta_k := \sum_{j=0}^k \alpha_j \) satisfies a linear recurrence of order 2 with characteristic polynomial $X^2-3X+1$.

It follows that
\[
R_b \geq \frac{C' R_a}{x_b \phi^b} \geq \frac{K}{b \phi^b}, \quad \text{for constants } C', K > 0.
\]
The second inequality follows from the asymptotic relation \(x_b \sim b\) as \(b \to \infty\) and from the fact that 
\(R_a \ge \prod_{k=1}^a \frac{1}{1+x_k} > 0,\)
which is independent of \(n\) (see inequality~\eqref{eq:boundR_k}, which directly follows from relation~\eqref{eq:relationmu_j}).

Now, choose \( b > a + 1 \) large enough so that \( \dfrac{K}{b \phi^b} = \dfrac{K \lambda_b}{y_b} \geq \dfrac{\sqrt{\lambda_b}}{y_b} \). Then set \( x_{b+1} = y_b > x_b \), so that
\[
\sum_{k=0}^n \mu_k x_k \geq R_b x_{b+1} \geq \sqrt{\lambda_b}, \quad \text{for all } n > b.
\]

In conclusion, \( \sum_{k=0}^n \mu_k x_k \to \infty \) as \( n \to \infty \), and the system is critical, while \( x_k \leq y_k \) for all \( k \ge 1\).
\end{proof}

\subsection{Characterization via the ratio \(x_{n+1}/x_n^2\)}
\label{subsec:characterization ratio}
First, equation~\eqref{eq:relationmu_j} gives the following bounds on \(\mu_k\), using the fact that the sequence \((x_k)\) is increasing:
\begin{equation}
x_k\Bigl(1 - \frac{x_k}{x_{k+1}}\Bigr) R_k \;\leq\; \mu_k \;\leq\; x_k R_k.
\end{equation}
As a consequence, we obtain bounds directly on \(R_k\) in terms of \(R_{k-1}\):
\begin{equation}
\frac{R_{k-1}}{1 + x_k} \;\leq\; R_k \;\leq\; \frac{R_{k-1}}{1 + x_k\Bigl(1 - \frac{x_k}{x_{k+1}}\Bigr)}.
\end{equation}

By induction on $k$, we then derive the following bounds for $R_k$ :
\begin{equation}
\label{eq:boundR_k}
\prod_{r=1}^{k} \frac{1}{1 + x_r}
\;\leq\;
R_k
\;\leq\;
\prod_{r=1}^{k} \frac{1}{1 + x_r\left(1 - \frac{x_r}{x_{r+1}}\right)}.
\end{equation}
First, we will look at the upper bound:
\Strongnonc*
\begin{proof}
By Lemma~\ref{lemma:5.2}, it suffices to determine conditions under which
\[
\sum_{k=1}^n R_k (x_{k+1} - x_k) \leq \sum_{k=1}^n R_k x_{k+1}
\]
remains bounded as \( n \to \infty \).

\noindent
By equation \eqref{eq:boundR_k}, we have
\[
\sum_{k=1}^n R_k x_{k+1} \leq \sum_{k=1}^\infty x_{k+1} \prod_{r=1}^k \frac{1}{1 + x_r \left(1 - \frac{x_r}{x_{r+1}}\right)}.
\]

Applying the d'Alembert ratio test, the series converges if
\begin{equation}
\label{convergencelimsup}
\limsup_{k \to \infty} \frac{x_{k+1}}{x_k \bigl(1 + x_k \bigl(1 - \frac{x_k}{x_{k+1}}\bigr)\bigr)} < 1.
\end{equation}

Fix \(x_k\) and define the function
\begin{equation}
f_k(x) \;=\; \frac{x}{x_k\bigl(1 + x_k(1 - \tfrac{x_k}{x})\bigr)}.
\end{equation}
A straightforward computation of \(f_k'(x)\) shows that \(f_k\) is strictly decreasing on \(\bigl(x_k+1,\tfrac{2x_k^2}{x_k+1}\bigr)\) and strictly increasing for \(x>\tfrac{2x_k^2}{x_k+1}\). In particular,
\[
f_k(x_k+1)
=\frac{{(x_k+1)}^2}{x_k(2x_k+1)}
\le\frac{9}{10}
<1
\]
for $k\geq1$, because $x_k \geq2$ and the expression is decreasing in $x_k$.

Hence, for any fixed $\alpha \in (0.9,1)$,
\[
f_k(x_{k+1}) \le \alpha \quad \Longleftrightarrow \quad x_{k+1} \le x_+(k,\alpha),
\]
where $x_+=x_+(k,\alpha)$ is the larger solution of $f_k(x_+)=\alpha$ meaning the larger root of
\[
 x^2 - \alpha x_k (1+x_k) x + \alpha x_k^3 = 0.
\]
An asymptotic expansion gives $x_+ \sim \alpha\, x_k^2$ as $k \to \infty$.  

Assume $\limsup_{k \to \infty} \frac{x_{k+1}}{x_k^2} < 1$. Then there exists $0.9 < \widetilde{\alpha} < 1$ such that, for $k$ large enough, $x_{k+1} \le \widetilde{\alpha} x_k^2$. Pick $\alpha$ satisfying $\widetilde{\alpha} < \alpha < 1$, and let $x_+(k,\alpha)$ be the corresponding root. For $k$ sufficiently large, $x_+ \ge \widetilde{\alpha} x_k^2 \ge x_{k+1}$, which implies $f_k(x_{k+1}) \le \alpha$ .  

Hence,
\[
\limsup_{k \to \infty} f_k(x_{k+1}) < 1.
\]

We have proved~\eqref{convergencelimsup}, so the model is non-critical.

\end{proof}

\textbf{Example~\ref{ex:noncritical}:}
\begin{itemize}
  \item If \( x_k \sim a^{b^k} \), with \( a > 1 \), \( 1 < b < 2 \), then :
  \[
  \frac{x_{k+1}}{x_k^2} \sim a^{b^k(b - 2)} \to 0.
  \]
  
  \item If \( x_k \sim \lambda a^{2^k} \), with \( a > 1 \), \( \lambda > 1 \), then :
  \[
  \frac{x_{k+1}}{x_k^2} \to \frac{1}{\lambda} < 1.
  \]
\end{itemize}

We now focus on the lower bound:
\Strongc*
\begin{proof} 
It follows from the lower bound~\eqref{eq:boundR_k} that
\[
R_k (x_{k+1}-x_k)
\;\ge\;
(x_{k+1}-x_k)\prod_{r=1}^{k}\frac{1}{1 + x_r}
\;\ge\;
x_{k+1}\prod_{r=1}^{k}\frac{1}{1 + x_r}
-\frac{1}{k!}.
\]
The result then follows directly from Lemma~\ref{lemma:5.2}.
\end{proof}

The d'Alembert ratio test gives the following corollary (using \(
\liminf \frac{x_{n+1}}{x_n(1+x_n)} = \liminf \frac{x_{n+1}}{x_n^2} 
\)).
\Weakc*

\Weakcthreshold*
\begin{proof}
We prove the contrapositive, using Theorem~\ref{thm:Strongc}: assume that
\[
\sum_{n=1}^{+\infty}\left(x_n \prod_{r=1}^{n-1} \frac{1}{1 + x_r} \right) < \infty.
\]
We prove that there exists \(a>1\) such that for all \(k\ge1\),
\[
x_k \le a^{2^k}.
\]

Since the series converges, there is \(k_0\ge1\) such that
\[
x_k \le \prod_{r=1}^{k-1}(1+x_r)\qquad\text{for all }k\ge k_0.
\]
Choose \(a\ge \sqrt{2}\) large enough so that \(x_k \le a^{2^k}\) for all \(k<k_0\). We prove by induction on \(k \ge k_0\) that \(x_k \le a^{2^k}\).

Assume \(x_r \le a^{2^r}\) for all \(r<k\). Then
\[
x_k \le \prod_{r=1}^{k-1}(1+x_r) \le \prod_{r=1}^{k-1}(1+a^{2^r}).
\]
Using the identity
\begin{equation}
\label{eq:product_identity}
\prod_{r=0}^{k-1} (1 + a^{2^r}) = \frac{a^{2^k} - 1}{a-1},
\end{equation}
we get
\[
\prod_{r=1}^{k-1} (1 + a^{2^r}) = \frac{1}{1+a} \prod_{r=0}^{k-1} (1 + a^{2^r}) = \frac{a^{2^k}-1}{(a-1)(1+a)}.
\]
Since \(a>\sqrt{2}\), we have \((a-1)(1+a) = a^2-1 > 1\), hence
\[
\frac{a^{2^k}-1}{(a-1)(1+a)} \le a^{2^k}.
\]
This proves \(x_k \le a^{2^k}\) and completes the induction.

\end{proof}

\textbf{Example~\ref{ex:critical}:}
\begin{itemize}
  \item If \( x_k \sim a^{b^k} \), with \( a > 1 \) and \( b > 2 \), then:
  \[
  \frac{x_{k+1}}{x_k^2} \sim a^{b^k(b - 2)} \to \infty.
  \]
  This results can also be recovered by Corollary~\ref{Weakcthreshold} and more generally, if $x_k \geq a^{b^k}$ from some index onward, then the system is critical.

  \item If \( x_k \sim \lambda a^{2^k} \), with \( a > 1 \) and \( \lambda < 1 \), then:
  \[
  \frac{x_{k+1}}{x_k^2} \to \frac{1}{\lambda} > 1.
  \]

\item Finally, assume 
\[
x_k = a^{2^k}(1 + \varepsilon_k) \quad \text{with} \quad \sum_{k=1}^\infty |\varepsilon_k| < \infty.
\]
The d'Alembert ratio test does not suffice, so we use Theorem~\ref{thm:Strongc}.

First, note that
\[
x_k \le a^{2^k}(1 + |\varepsilon_k|), \quad \text{hence} \quad 1 + x_k \le (1 + a^{2^k})(1 + |\varepsilon_k|).
\]

Using identity~\eqref{eq:product_identity}, we obtain
\[
\prod_{r=1}^{n-1} (1 + x_r) \le \prod_{r=1}^{n-1} (1 + a^{2^r}) \prod_{r=1}^{n-1} (1 + |\varepsilon_r|) 
= \frac{a^{2^n}-1}{(a-1)(1+a)} \prod_{r=1}^{n-1} (1 + |\varepsilon_r|).
\]

Since \(1 + |\varepsilon_r| \le e^{|\varepsilon_r|}\), it follows that
\[
\prod_{r=1}^{n-1} (1 + x_r) \le \frac{a^{2^n}}{a^2-1} \, e^{\sum_{r=0}^\infty |\varepsilon_r|}.
\]

Taking the reciprocal of the product, and multiplying by $x_n = a^{2^n}(1+\eps_n)$ gives
\[
x_n \prod_{r=1}^{n-1} \frac{1}{1 + x_r} \ge (a^2-1)(1 + \varepsilon_n) e^{-\sum_{r=0}^\infty |\varepsilon_r|}.
\]

This ensures
\[
\sum_{n=1}^\infty \left(x_n \prod_{r=1}^{n-1} \frac{1}{1 + x_r}\right) = \infty.
\]

\end{itemize}

\section*{Acknowledgments}
This work was carried out as part of an internship under the supervision of Ellen Saada. I am deeply grateful to her for her continuous guidance, support, and careful reading throughout the preparation of this work. I would also like to thank Frank Redig and Berend van Tol for meeting with me and for the valuable discussions we had about the problem.

\printbibliography

\end{document}